\documentclass[12pt,reqno]{amsart}
\usepackage[T1]{fontenc}   
\usepackage{comment}

\usepackage{amsmath,amsfonts,amsthm,mathrsfs,amssymb,amscd,enumerate,url,tikz-cd, enumitem}

\input{xypic}
\xyoption{all}

\usepackage{xcolor}
\colorlet{mdtRed}{red!50!black}
\definecolor{dblue}{rgb}{0,0,.6}
\usepackage[colorlinks]{hyperref}
\hypersetup{linkcolor=blue,citecolor=dblue,filecolor=dullmagenta,urlcolor=mdtRed}

\newtheorem{theorem}{Theorem}[section]
\newtheorem{proposition}[theorem]{Proposition}
\newtheorem{lemma}[theorem]{Lemma}
\newtheorem{corollary}[theorem]{Corollary}
\numberwithin{equation}{theorem}

\theoremstyle{definition}
\newtheorem{definition}[theorem]{Definition}

\newtheorem{rmk}[theorem]{Remark}
\newtheorem{example}[theorem]{Example}

\newcommand{\mb}[1]{\mathbb{#1}}
\newcommand{\mc}[1]{\mathcal{#1}}
\newcommand{\F}{\mathbb{F}}
\newcommand{\mL}{\mathcal{L}}
\newcommand{\mO}{\mathcal{O}}
\newcommand{\rar}{\rightarrow}
\renewcommand{\t}[1]{\widetilde{#1}}
\newcommand{\Con}{\operatorname{Con}}
\newcommand{\Div}{\operatorname{Div}}

\title{Equality and iso-duality of Generalized Geometric Goppa codes}
\author{Jyoti Dasgupta}
\author{Nupur Patanker}

\begin{document}

\begin{abstract}
Geometric Goppa codes are constructed by evaluating elements of an algebraic function field at the places of degree one. Extending this, generalized geometric Goppa codes were defined in \cite{base} and \cite{GGG} by allowing evaluation at places of arbitrary degree. In this paper, we establish necessary and sufficient conditions for equality of generalized geometric Goppa codes, following the approach of \cite{MP}. These conditions serve as a key tool in our characterization of iso-dual and self-dual generalized geometric Goppa codes. Given a separable extension $M/F$ of algebraic function fields and an iso-dual generalized geometric Goppa code over $F$, we extend the construction of \cite{CPQT} to obtain iso-dual generalized geometric Goppa codes over $M$. Furthermore, we study self-duality of generalized AG codes (\cite{base}, \cite{DM}).
\end{abstract}

\maketitle

\section{Introduction}
Let $\mathbb{F}_q$ denote a finite field with $q$ elements. An $[n,k,d]$-code over $\mathbb{F}_q$ is defined as a $k$-dimensional subspace of $\mathbb{F}_q^n$ over $\mathbb{F}_q$ with minimum distance $d$. Geometric Goppa codes form a class of linear codes constructed using tools from algebraic geometry. In this construction introduced by Goppa \cite{Goppa}, one start with an algebraic function field $F/\mathbb{F}_q$, rational places $P_1, \dots, P_n$ and a divisor $G$ of  $F/\mathbb{F}_q$ with support disjoint from $D=P_1+ \dots+P_n$. Then the geometric Goppa code is obtained by evaluating elements of the Riemann-Roch space $\mathcal{L}(G)$ at the places $P_1,\dots,P_n$. Xing, Niederreiter and Lam in \cite{base} generalized this construction of geometric Goppa codes by considering places of higher degree and called this broader class of linear codes as Generalized AG codes. The construction is given as follows:\par
 Let $P_1, \dots, P_s$ be distinct places of $F$ such that $\deg~P_i=k_i$ for $i=1,2,\dots,s$. Let $G$ be a divisor of $F$ whose support is disjoint from $D:=P_1+\dots,+P_s$. For $i=1,\dots,s$, $\mathcal{O}_{P_i}$ denotes the valuation ring of $P_i$ and for $x \in \mathcal{O}_{P_i}$, $x(P_i)$ denotes the residue of $x$ in $\mathcal{O}_{P_i}/P_i$. For each $i=1,\dots,s$, a linear $[n_i,k_i,d_i]$-code $C_i$ over $\mathbb{F}_q$ is chosen. Let $\pi_i: \mathcal{O}_{P_i}/P_i \rightarrow C_i$ be a chosen $\mathbb{F}_q$-linear isomorphism, $i=1,\dots,s$. Then the generalized Geometric AG code is defined as the image of the map
$$
    \alpha: \mathcal{L}(G) \xrightarrow{ev} \prod_{i=1}^s \mathcal{O}_{P_i}/P_i \xrightarrow{\prod_{i=1}^s \pi_i} \prod_{i=1}^s C_i,~x \mapsto (x(P_i))_{i=1}^s \mapsto (\pi_i(x(P_i)))_{i=1}^s.
$$
It is denoted by $C_{\mathcal{L}}(P_1,\dots,P_s;G;C_1,\dots,C_s)$. The authors in \cite{base} also gave examples of many good linear codes obtained from this construction. In \cite{GGG}, the author considered the image of the first map $ev$ and called it Generalized Geometric Goppa code denoted by $C_{\mathcal{L}}(D,G)$.\par 
Since then, many researchers have studied different properties of these codes. Duality of GGG codes and GAG codes have been studied in \cite{GGG} and \cite{DM}. Iso-duality of GAG codes has been studied in \cite{isodualgag}. The relation between AG codes and GAG codes has been studied in \cite{AGGAG}. Automorphisms of GAG codes have been studied in \cite{automorphism}, \cite{automorphism1}, \cite{automorphism2}, \cite{automorphism3}, etc. Other works on GAG codes can be found in \cite{decode}, \cite{decode1}, \cite{decode2}, \cite{decode3}, \cite{locality}, \cite{asymptotic},  \cite{Roch}, \cite{mindis}, \cite{mindis1}, \cite{GAG1}, \cite{GAG2}, \cite{convolutional}, \cite{binary}, \cite{evaluation}, \cite{ff}, etc.

For a linear code $C \subseteq \mathbb{F}_q^n$ and an inner product on $\mathbb{F}_q^n$, the orthogonal complement $C^{\perp}$ of $C$ with respect to the inner product is called the dual code of $C$. If $C=C^{\perp}$, then $C$ is called self-dual code. More generally, if $C$ is equivalent to $C^{\perp}$, then we call it iso-dual code. Iso-dual AG codes have been studied in \cite{isodual}, \cite{isodual1}, \cite{isodual2}, \cite{isodual3}, etc. Given a finite separable extension $M/F$ and an iso-dual geometric Goppa code over $F$, the authors in \cite{CPQT} constructed iso-dual geometric Goppa codes over $M$ under certain conditions. In this paper, we study the equality and iso-duality of generalized geometric Goppa codes. We mimic the approach of \cite{CPQT} to construct iso-dual generalized geometric Goppa codes over function field extensions.\par
This paper is organised as follows. Section $2$ recalls the basics of algebraic function fields and generalized geometric Goppa codes. Section $3$ examines the conditions under which two generalized geometric Goppa codes are equal using results from \cite{MP}. Section $4$ extends the results on iso-duality of generalized geometric Goppa codes and establishes the necessary and sufficient conditions for a generalized geometric Goppa code to be iso-dual. Section $5$ deals with the construction of iso-dual generalized geometric Goppa codes over a finite separable extension $M/F$ of algebraic function fields using iso-dual generalized geometric Goppa codes over $F$. Section $6$ provides a few examples of the construction of Section $5$. Section $7$ deals with the condition for self-duality of  generalized AG codes.

\section{Preliminaries}
In this section, we recall in brief some basic concepts from algebraic function field and generalized geometric Goppa codes. For more details, refer to \cite{St}, \cite{DM}.\par

\subsection{Algebraic function fields}
Let $K=\F_q$ be the base field and $F/ K$ be an algebraic function field. We denote the set of places of $F$ by $\mathbb{P}(F)$. For a place $P \in \mathbb{P}(F)$, let $v_P$ denote the discrete valuation associated with $P$, $\mathcal{O}_{P}$ denote the valuation ring of $P$ and $F_P$ denote the residue class field $\mathcal{O}_{P}/P$. For $x \in \mathcal{O}_{P}$, the residue of $x$ in $F_{P}$ is denoted by $x(P)$. We denote the group of divisors of $F$ by $\Div (F)$. For a rational function $x \in F$, the associated principal divisor is denoted by $(x)$.  We have the following equivalence relations on $\Div (F)$.
\begin{definition}
    Two divisors $G$ and $H$ in $\Div F$ are said to be (rationally) equivalent if there exists
a rational function $u\in F$ such that $H = G + (u )$. We denote this by $G \sim H$.
The divisors $G$ and $H$ in $\Div F$ are said to be (rationally) equivalent with respect to a divisor $D$ if, moreover,
$u(P) = 1$, for all $P \in \textrm{Supp }(D)$. We denote this by $G \sim_D H$ (see \cite[Definition 2.15]{MP}).
\end{definition}
 The divisor class group of $F$ is denoted by $\textrm{Cl}(F)$ and defined as $\Div(F)$ modulo the equivalence relation $\sim$.

\subsubsection{Extensions of function fields}
Let $F'/K'$ be a finite extension of $F/K$. For a place
$P \in \mb{P}_F$, its \emph{conorm} (with respect to $F'/F$) is defined as
\begin{equation}\label{eq:place_conorm}
    \Con_{F'/F}(P) := \sum_{P' \mid P} e(P'|P)  P',
\end{equation}
where the sum runs over all places $P' \in \mb{P}_{F'}$ lying over $P$ and $e(P'|P)$ is the ramification index of $P'$ over $P$. The conorm map
is extended to a homomorphism from $\Div(F)$ to $\Div(F')$ by linearity.
If $h \in F$, then we have that 
\begin{equation}\label{eq:prin_conorm}
\Con_{F'/F} \left((h)^F \right)=(h)^{F'},
\end{equation}
where $(h)^F$ denotes the principal divisor of $h$ in $F$, and similarly, $(h)^{F'}$ denote the principal divisor of $h$ in $F'$. Thus, the conorm map induces a map of the divisor class groups
\begin{equation}\label{eq:cl_conorm}
\Con_{F'/F}: \textrm{Cl}(F) \rar \textrm{Cl}(F').
\end{equation}
For a divisor $A \in \Div(F)$, let us denote its image in $\Div(F')$ by $\t{A}:=\Con_{F'/F}(A)$. Then, by \cite[Corollary 3.1.14]{St}, we know that 
\begin{equation}\label{eq:degree}
     \deg \t{A}=\frac{[F':F]}{[K':K]}\deg A.
\end{equation}
There is also a natural divisor associated with the extension, defined by
\begin{equation}\label{eq:different}
  \textrm{Diff}(F'/F) := \sum_{P \in \mb{P}_F}\sum_{P' \mid P} d(P'|P)  P',  
\end{equation}
called the different of $F'/F$, where $d(P'|P)$ is the different exponent of $P'$ over $P$.
Let $\Omega_F$ denote the module of Weil differentials of $F$ over $K$.  Let $(\omega) \in \Div F$ be the divisor associated with a Weil differential $\omega \in \Omega_F$. There is a well-defined map $\textrm{Cotr}_{F'/F}: \Omega_F \rar \Omega_{F'}$ from the module of Weil differentials of $F$ over $K$ to the module of Weil differentials of $F'$ over $K'$, such that
\begin{equation}\label{eq:cotrace}
    (\textrm{Cotr}_{F'/F}(\omega)) = \Con_{F'/F}((\omega)) + \textrm{Diff}(F '/F).
\end{equation}

\subsection{Generalized geometric Goppa codes}
In this subsection, we recall the construction of generalized geometric Goppa codes and some relevant results.

Let $P_1, \dots, P_s$ be distinct places of $F$ such that $\deg P_i=k_i$ for $i=1,2,\dots,s$. Let $\mathcal{V}$ be the $\mathbb{F}_q$-vector space $\prod_{i=1}^s F_{P_i}$ with dimension $ \sum_{i=1}^s k_i$. 
Consider the divisor $D:=P_1+\dots+P_s$ and another divisor $G$ of $F$ whose support is disjoint from $D$. The Riemann-Roch space associated to $G$ is $\mL(G)=\{x \in F~|~ (x) +G \geq 0\}$. We denote by $l(G)$ the dimension of $\mL(G)$ as an $\F_q$-vector space. Recall that the generalized geometric Goppa code is the image of the evaluation map
\begin{equation} \label{GGG_1}
 ev: \mathcal{L}(G) \rightarrow \mathcal{V},~x \mapsto (x(P_i))_{i=1}^s.   
\end{equation}
For a place $P \in \mb{P}(F)$ and a Weil differential $\omega \in \Omega_F$, denote the residue of $\omega$ at $P$ by $\textrm{res}_P(\omega)$. In \cite{DM}, the authors have also defined generalized geometric Goppa codes using differentials. We recall the construction now. Given a divisor $A \in \Div F$, we have the set of Weil differentials $\Omega(A)=\{\omega \in \Omega ~|~ (\omega) \geq A\}$. 
Consider the map 
$$\textrm{rs}: \Omega(G-D) \rightarrow \mathcal{V}, \omega \mapsto (\textrm{res}_{P_i}(\omega))_{i=1}^s.$$
As defined in \cite[Proposition 2]{DM}, we have the following inner product on $\mathcal{V}$: for $x=(x_1,\dots,x_s)$, $y=(y_1,\dots,y_s) \in \mathcal{V}$, 
\begin{equation} \label{eq_inn}
    \langle x,y \rangle:=\sum_{i=1}^s \textrm{Tr}_{F_{P_i}/ \mathbb{F}_q}(x_iy_i).
\end{equation}
For a linear code $C \subseteq \mathcal{V}$, let $C^{\perp}$ denote the dual code of $C$ with respect to the inner product (\ref{eq_inn}). Then from \cite[Proposition $3$]{DM}, we have $\dim~C+\dim~C^{\perp}=\dim~\mathcal{V}$ and $C^{\perp \perp}=C$.\par
Moreover, $C$ is said to be iso-dual if there exists $x =(x_1, \ldots, x_s) \in \mc{V}$  where $x_i \neq 0~\forall~ 1 \leq i \leq s$ such that $ C^\perp = x \cdot C$ and is said to be self-dual if $C=C^\perp$.\\

We recall the following results from \cite{DM}, which will be used later. 
\begin{theorem}\label{th:differential} Using the inner product defined in (\ref{eq_inn}), the following hold:
    \begin{enumerate}
\item \cite[Theorem $4$]{DM} $C_{\Omega}(G,D)=C_{\mathcal{L}}(G,D)^{\perp}.$
        \item \cite[Theorem $5$]{DM} Let $\eta$ be a differential of $F$ with $v_{P_i}((\eta))=-1$ $(i=1,\dots,s)$. Then $$C_{\Omega}(D,G)=a \cdot C_{\mathcal{L}}(D,D-G+(\eta))$$
        with $a=(\textrm{res}_{P_1}(\eta), \dots, \textrm{res}_{P_s}(\eta))$.
        \item \cite[Corollary $1$]{DM} There exists a differential $\eta$ with $v_{P_i}((\eta))=-1$ and $res_{P_i}(\eta)=1$ for $i=1,\dots,s$  such that 
        $$C_{\Omega}(D,G)=C_{\mathcal{L}}(D,D-G+(\eta)).$$
    \end{enumerate}
\end{theorem}

 \section{Equality of generalized geometric Goppa codes}

In this section, we characterize divisors giving the same generalized geometric Goppa codes,  extending \cite[Theorem 4.14]{MP}.  


We follow the notations introduced before. Let the base field be denoted by $K=\mathbb{F}_q$. Since $F/K$ is an algebraic function field, there exists an element $x$ of $F$, transcendental over $K$, such that $F$ is a finite extension of $K(x)$. We may assume that the fields $F_{P_i}$ for all $i=1, \ldots, s$ and $F$ are contained in some common field extension of $K$ (e.g. consider the algebraic closure of $K(x)$). Let $K'$ be the compositum of the fields $F_{P_i}$ for all $i=1, \ldots, s$. Then $K'$ is a finite extension of $K$.
Thus, the compositum $F'=FK'$ is a function field with the full constant field $K'$ (see \cite[Proposition 3.6.1]{St}). Then, by \cite[Theorem 3.6.3]{St}, we know that $F'/F$ is unramified, which means that $\Con_{F'/F}(P) = \sum_{P' \mid P} P'$ (see \eqref{eq:place_conorm}). Furthermore, for any divisor $A \in \Div(F)$, by \eqref{eq:degree}, we have $\deg A= \deg \t{A}$.

The upshot of moving to the extension field $F'$ is the following lemma.
\begin{lemma}\label{lem:lift_D}
    $\t{D}=\Con_{F'/F}(D)=\sum_{i=1}^s\sum_{P' \mid P_i} P'$ is a divisor in $F'$ with all rational places.
\end{lemma}

\begin{proof}
    By definition, $k_i=[F_{P_i}:K]$ divides $[K':K]$ for all $i=1, \dots, s$. If $P' \in \mb{P}(F')$ lies over $P_i$ for some $i$ between $1$ to $s$, then $\deg P'=\frac{k_i}{\text{gcd}(k_i, [K':K])}=1$ (see \cite[Theorem 1.7.2(b)]{NX}). This finishes the proof.
\end{proof}

The following result is one of the key ingredients in our proof of Theorem \ref{th:equality}. We include the proof for completeness.
\begin{lemma}\label{equi}
     If $G, H, D$ are divisors in $F$ such that $\t{G}, \, \t{H}$ are rationally equivalent with respect to $\t{D}$, then $G, \, H$ are rationally equivalent with respect to $D$. 
\end{lemma}

\begin{proof}
    Since $\t{G}, \, \t{H}$ are rationally equivalent with respect to $\t{D}$, there exists $\t{h} \in F'$ such that 
    \begin{equation}\label{values}
        \t{G}-\t{H}=(\t{h})^{F'} ~\textrm{and} ~\t{h}(P')=1 \textrm{ for all }P' \in \textrm{Supp } (\t{D}). 
    \end{equation}
Furthermore, since the induced conorm map (see \eqref{eq:cl_conorm}) of the divisor class groups
\(\Con_{F'/F}: \textrm{Cl}(F) \rar \textrm{Cl}(F')\) is injective (by \cite[Theorem 3.6.3 (f)]{St}), it follows that $G$ is rationally equivalent to $H$. Hence, there is $h \in F$ such that
\begin{equation*}
        \begin{split}
            & G-H= (h)^F\\
             \Rightarrow & \t{G}-\t{H}=\textrm{Con}_{F'/F}((h)^F)\\
             \Rightarrow &(\t{h})^{F'}=(h)^{F'} ~ (\textrm{by \eqref{values}}).
        \end{split}
    \end{equation*}
This implies that there exists a nonzero element $c \in K'$ such that 
\begin{equation*}
    \t{h}=ch.
\end{equation*}
We want to show that $c \in K$. From \eqref{values}, we have that 
\begin{equation}\label{eq c}
    c^{-1}=h(P') \in F_{P'}
\end{equation}
 for all extensions $P'$ of $P \in \textrm{Supp}~(D)$.\par 

Next, note that $K'/K$ is a finite separable extension, hence a Galois extension. This implies that $F'/F$ is also Galois and the map $$\textrm{Gal}(F'/F) \rightarrow \textrm{Gal}(K'/K),$$ which maps $\sigma$ to its restriction $\sigma|_{K'}$, is an isomorphism. Thus, for $\rho \in \textrm{Gal}(K'/K)$, there exists a unique $\t{\rho} \in \textrm{Gal}(F'/F)$ such that $\t{\rho}|_{K'}=\rho$. Fix $P \in \textrm{Supp}(D)$ and $P'$ be an extension of $P$ in $F'$. Since $\t{\rho}$ acts transitively on the set of extensions of $P$, we see that $\t{\rho}(P') \in \textrm{Supp}~(\t{D})$. Thus, by \cite[Lemma 3.5.2]{St} and \eqref{eq c}, we get
$$\rho(c^{-1})=\t{\rho}(h(P'))=\t{\rho}(h)(\t{\rho}(P'))=h(\t{\rho}(P'))=c^{-1}.$$
This shows that $c \in K$ and hence $\t{h}=ch \in F$. This implies that $\t{h}(P)=\t{h}(P')=1$ and $G-H=(\t{h})^F$. Thus, $G$ and $H$ are equivalent with respect to $D$.
\end{proof}

The following lemma will be used later.
\begin{lemma}\label{lem:Conorm}
Let $P$ be a rational place in $\mb{P}(F)$ such that $\Con_{F'/F}(P)$ is rationally equivalent to $P'$, a rational place in $\mb{P}(F')$. If genus of $F$ (and hence of $F'$) is greater than zero, then $\Con_{F'/F}(P)=P'$. 
\end{lemma}

\begin{proof}
    By definition, $\Con_{F'/F}(P) = \sum_{P'' \mid P} P''$. Since $\deg \Con_{F'/F}(P)= \deg P =1$, we get that $\Con_{F'/F}(P)=P''$ for some rational place $P'' \in \mb{P}(F')$ such that $P''$ lies over $P$. By hypothesis, the genus of $F'$ is greater than zero, which implies that $1$ is a gap number for $F'$. Since the rational place $P''$ is equivalent to $P'$ in $F'$, it follows that $P''=P'$.
\end{proof}

The following theorem is a generalisation of \cite[Theorem 4.14]{MP}.

\begin{theorem}\label{th:equality}
   Let $D=\sum_i^sP_i$ be as before. Suppose $\deg D > 2g + 2$. Let $G$ and $H$ be two divisors of the same
degree $m$ on a curve of genus $g$, with support disjoint from $D$. If $C_{\mL}(D,G) \neq \{\mathbf{0}\}$ and $C_{\mL}(D,G) \neq \mc{V}$, and
$2g - 2 < m < \deg D$, then $C_{\mL}(D,G)=C_{\mL}(D,H)$ if and only if:
\begin{itemize}
    \item[(a)] $G$ and $H$ are equivalent with respect to $D$, or
    \item[(b)] there exist two rational points $P$ and $Q$ such that $G - P$ and $H - Q$ are
    equivalent with respect to $D$ and are canonical divisors, or
    \item[(c)] there exist two rational points $P$ and $Q$ such that $G + P$ and $H + Q$ are
    equivalent with respect to $D$ and are equivalent with $D$, or
    \item[(d)] there exists an $i$, $1 \leq i \leq s$, such that $P_i$ is a rational place and $G - P_i$ and $H - P_i$ are equivalent
    with respect to $D - P_i$ and are canonical divisors, or
    \item[(e)] there exists an $i$, $1 \leq i \leq s$, such that $P_i$ is a rational place and $G + P_i$ and $H + P_i$ are equivalent
    with respect to $D - P_i$ and are equivalent with $D$.
\end{itemize}
\end{theorem}

\begin{proof} It is easy to see that if any one of $(a),~(b),~(c),~(d)$ or ~ $(e)$ is satisfied then $C_{\mL}(D,G)=C_{\mL}(D,H)$ (follows by same arguments as in \cite[Example $4.3$]{MP}). Now we prove the converse direction.\par
Since every basis of $\mL({A})$ is also a basis of $\mL(\t{A})$ for any $A \in \Div(F)$, equality of the codes $C_{\mL}(G,D)=C_{\mL}(H,D)$ will imply that $C_{\mL}(\t{G},\t{D}, K')=C_{\mL}(\t{H},\t{D},K')$. Then by Lemma \ref{lem:lift_D} and \cite[Theorem 4.14]{MP}, we have the following situations.
    \begin{itemize} 
        \item[Case (I):]{\bf $\t{G}, \, \t{H}$ are rationally equivalent with respect to $\t{D}$.} 
        
        This implies that $G, \, H$ are rationally equivalent with respect to $D$ (from Lemma \ref{equi}). This proves (a).

        \item[Case (II):]{\bf There exist two rational points $P'$ and $Q'$ in $\mb{P}(F')$ such that $\t{G} - P'$ and $\t{H} - Q'$ are equivalent with respect to $\t{D}$ and are canonical divisors.}
        
        We will show that there is a rational place $P \in \mb{P}(F)$ such that $\Con_{F'/F}(P)=P'$ in $\Div(F')$. We may choose a canonical divisor $W$ in $\Div(F)$, which implies that $\t{W}$ is a canonical divisor in $\Div(F')$. Since $\t{G}-P'$ is canonical, it is rationally equivalent to $\t{W}$, whence $\t{G}- \t{W}$ is rationally equivalent to $P'$. This means that \[l(G-W, K)= l(\t{G}- \t{W}, K')= l(P',K')>0.\] Thus, there is a rational function $v \in F$ such that $(v) + G -W$ is an effective divisor in $\Div(F)$. Observe that \[\deg ((v) + G -W)= \deg (\t{G}- \t{W})=\deg P'=1.\] Hence, the effective divisor $(v) + G - W$ must be a rational place $P \in \mb{P}(F)$. Then, by definition, we have that \[\Con_{F'/F}(P)= \Con_{F'/F}((v) + G -W)= \Con_{F'/F}((v)) + \t{G}- \t{W} \sim \t{G}- \t{W} \sim P'.\]  
        Now, if $g=0$ then
        \begin{align*}
            & \deg((v) + G -W)=1 \implies \deg(G)-\deg(W)=1 \\
            \implies & \deg(G)+2=1 \implies \deg(G)=-1<0\\
            \implies & l(G)=0 \implies C_{\mL}(D,G)=\{\mathbf{0}\}, \text{ contradiction}
        \end{align*}
        Hence, $g>0$ and by Lemma \ref{lem:Conorm}, we know that $\Con_{F'/F}(P)=P'$. Similarly, we can show that there is a rational place $Q \in \mb{P}(F)$ such that $\Con_{F'/F}(Q)=Q'$. Thus, we now have that the canonical divisors $\t{G} - \Con_{F'/F}(P)$ and $\t{H} - \Con_{F'/F}(Q)$ are equivalent with respect to $\t{D}$, where $P, Q$ are rational places in $\Div(F)$. By Lemma \ref{equi}, we get that $G - P$ and $H - Q$ are equivalent with respect to $D$. Hence, in this case,  the statement (b) holds. 

        \item[Case (III):]{\bf There exist two rational points $P'$ and $Q'$ in $\mb{P}(F')$ such that $\t{G} + P'$ and $\t{H} + Q'$ are equivalent with respect to $\t{D}$ and are equivalent with $\t{D}$.} 

        We have the following
        \begin{equation*}
            \begin{split}
                \t{G} + P' \sim \t{D} &\Rightarrow \Con_{F'/F} (D-G) \sim P'\\
                & \Rightarrow l(D-G, K)= l (\Con_{F'/F} (D-G), K')= l (P',K')>0.
            \end{split}
        \end{equation*}

        As before, there is a rational function $v \in F$ such that $(v) + D -G$ is an effective divisor of degree one in $\Div(F)$, hence given by a rational place $P$ of $F$. Now, we have 
        \begin{equation*}
            \Con_{F'/F}(P)= \Con_{F'/F}((v) + D -G)= \Con_{F'/F}((v)) + \t{D}- \t{G} \sim \t{D}- \t{G} \sim P'.
        \end{equation*}
         Now if $g=0$ then
        \begin{align*}
            & \deg((v) + D -G)=1 \implies \deg(D)- \deg(G)=1 \\
            \implies & \deg(G)= \deg(D)-1 \implies l(G)= \deg(G)+1= \deg(D)\\
            \implies & C_{\mL}(D,G)=\mc{V}, \text{ contradiction}
        \end{align*}
        Hence, $g>0$ and by Lemma \ref{lem:Conorm}, we know that $\Con_{F'/F}(P)=P'$. Similarly, there is a rational place $Q \in \mb{P}(F)$ such that $\Con_{F'/F}(Q)=Q'$. Thus, we now have that $\t{G} +\Con_{F'/F}(P)$ and $\t{H} +\Con_{F'/F}(Q)$ are equivalent with respect to $\t{D}$ and are equivalent with $\t{D}$, where $P, Q$ are rational places in $\Div(F)$. Then, we get (c) by Lemma \ref{equi}.

        \item[Case (IV):]{\bf There exists $P' \in \textrm{Supp }(\t{D})$, such that $\t{G} - P'$ and $\t{H} - P'$ are equivalent
    with respect to $\t{D} - P'$ and are canonical divisors.} 

    Let us write $\t{D}= \sum_{j=1}^t P'_j$, where $\deg P'_j=1$ for all $j=1, \ldots, t$. As in case (a), we choose a canonical divisor $W$ in $\Div(F)$, so that $\t{W}$ is a canonical divisor in $\Div(F')$. Then 
    \[\t{G}- P'_j \sim \t{W} \Rightarrow \t{G} - \t{W} \sim P'_j, \]
    which implies as above that there is a rational function $v \in F$ such that $(v) + G -W$ is an effective divisor of degree one in $\Div(F)$, hence given by a rational place $P$ of $F$. Again, by Lemma \ref{lem:Conorm} and since $g>0$,  we know that $\Con_{F'/F}(P)=P_j'$, which means that $P'_j$ lies over the rational place $P$. However, since $P'_j \in \text{Supp}(\t{D})$, we know that $P'_j$ lies over some place in the support of $D$. By uniqueness, $P$ must belong to the support of $D$ (see \cite[Proposition 3.1.7]{St}), say $P=P_i$ for some $i, \, 1 \leq i \leq \  s$. Thus, we have that $\t{D}- P'_j=\Con_{F'/F}(D-P_i)$. Thus, we get the statement in (d) by Lemma \ref{equi}.

      \item[Case (V):]{\bf There exists a $P' \in \textrm{Supp }(\t{D})$, such that $\t{G} + P'$ and $\t{H} + P'$ are equivalent
    with respect to $\t{D} - P'$ and are equivalent with $\t{D}$.} 

    Following arguments similar to the above, we would get (e).
    \end{itemize}
\end{proof}

 \section{Iso-duality of Generalized Geometric Goppa codes}

In this section, we characterize iso-dual generalized geometric Goppa codes with respect to the inner product on $\mc{V}$ defined in Equation  (\ref{eq_inn}). We follow the approach of \cite{MP} and use Theorem \ref{th:equality} and results from \cite{DM}. This characterization generalizes Proposition 2.6 of \cite{CPQT} (see also \cite{corrig}).

\begin{theorem}\label{th:iso-dual}
    Let $C_{\mL}(D, G)$ be a generalized geometric Goppa code defined over a function field $F$ of genus $g$. Assume that $2g-2< \deg G < \deg D$ and $\deg D > 2g+2$.

\begin{itemize}
    \item[(a)]Assume that, there exists a Weil differential $\eta$ such that $(\eta)=2G - D$ and $v_{P_i}((\eta))=-1$. Then $C_{\mL}(D, G)$ is iso-dual. More precisely,
    \[
    C_{\mL}(D, G)^\perp = x \cdot C_{\mL}(D, G)
    \]
    where $x = (\operatorname{res}_{P_1}(\eta), \ldots, \operatorname{res}_{P_s}(\eta))$.  
    
    \item[(b)] Conversely, if $C_{\mL}(D, G)$ is iso-dual then $2G - D$ is a canonical divisor. In this case, $\deg G= \frac{1}{2}(\deg D+2g-2)$.

    \item[(c)] If $F$ is a rational function field, then $C_{\mL}(D, G)$ is iso-dual if and only if $\deg G= \frac{1}{2}(\deg D-2)$.
\end{itemize}

\end{theorem}
\begin{proof}
\begin{itemize}
    \item[(a)]  This is precisely \cite[Theorem 6]{DM}.\\

    \item[(b)] Let $C_{\mL}(D, G)$ be iso-dual. In other words, there exists $a=(a_1, \ldots, a_s) \in \mc{V}$ with each coordinate nonzero, such that 
    \begin{equation}\label{eq:iso-dual}
        C_{\mL}(D, G)^\perp = a \cdot C_{\mL}(D, G).
    \end{equation}
 Following the proof of \cite[Theorem 5.5]{MP}, we want to find $h \in F$ such that $$h(P_i)=a_i$$ for all $i=1, \ldots, s$.\par
 Observe that, for each $i \in \{1, \ldots, s\}$, we have a lift $\tilde{a_i} \in \mc{O}_{P_i} \subset F$ of $a_i \in F_{P_i}$. Then by the Weak Approximation Theorem \cite[Theorem $1.3.1$]{St}, there exists $h \in F$ such that $$v_{P_i}(h-\tilde{a_i})=1$$ for all $i=1, \ldots, s$. For this choice of $h \in F$, we have $h-\tilde{a_i} \in P_i$ for all $i=1, \ldots, s$, which in turn implies that $$h(P_i)=\tilde{a_i}(P_i)=a_i$$ for all $i=1, \ldots, s$.

  On the other hand, by Theorem \ref{th:differential}, we know that 
    \begin{equation}\label{eq:perp_and_diff}
        \begin{split}
            C_{\mL}(D, G)^\perp&= C_{\Omega}(D,G) \\
            &=C_{\mL}(D, D-G+(\eta)),
        \end{split}
    \end{equation}
  where $\eta$ is a differential of $F$ with $v_{P_i}((\eta))=-1$ and $\operatorname{res}_{P_i}(\eta)=1$ for all $i=1, \ldots, s$. Let us denote the divisor $D-G+(\eta)$ by $G^{\perp}$.  
 Thus, by \eqref{eq:iso-dual} and \eqref{eq:perp_and_diff}, we get that 
 \[a \cdot C_{\mL}(D, G)= C_{\mL}(D, G)^\perp= C_{\mL}(D,G^\perp)=a \cdot C_{\mL}(D,G^\perp + (h)) . \] In other words, denoting the divisor $G^\perp + (h)$ by $H$, we have
 \begin{equation} \label{eq_equal}
      C_{\mL}(D, G)=  C_{\mL}(D,H) . 
 \end{equation}  
 Given the code $C_{\mL}(D, G)$ is iso-dual, we have that 
 \begin{equation}\label{eq:dim}
      \dim C_{\mL}(D, G)= \frac{1}{2}\deg D. 
 \end{equation} In particular, $C_{\mL}(D, G)$ is nontrivial.
Moreover, the kernel of the evaluation map in \eqref{GGG_1} is $\mL(G-D)$. Since $\deg G < \deg D$, the evaluation map is injective. Hence,
\begin{equation}\label{eq:global}
       \dim C_{\mL}(D, G)=l(G)=\deg G +1-g. 
 \end{equation} Then \eqref{eq:dim} and \eqref{eq:global} implies that
 \begin{equation}\label{eq:deg}
     \begin{split}
        \deg H & =\deg D-\deg G+\deg ((\eta))\\
        &=2 \deg G+2-2g-\deg G +2g-2\\
        &=\deg G. 
     \end{split}
 \end{equation}
 We have the assumptions $2g-2 <\deg G < \deg D$ and $\deg D>2g+2$. By \eqref{eq_equal} and \eqref{eq:deg}, we can apply Theorem \ref{th:equality}. Thus, one of the five situations given in the statement of Theorem \ref{th:equality} may occur. We rule out every possibility other than the one that $G$ and $H$ are equivalent with respect to $D$. To see this, note that in the case of (b) or (d) of Theorem \ref{th:equality}, $\deg G=2g-1$. Then by Riemann-Roch theorem, we get $\deg D=2g$, a contradiction. In the case of (c) or (e), we have $\deg G=\deg D-1$. Again, using Riemann-Roch theorem, we get $\deg D=2g$. Thus, the only possibility is when $G$ and $H$ are equivalent with respect to $D$. In other words, there is $u \in F$ such that $u(P_i)=1$ for all $i=1, \ldots, s$ and 
 \begin{equation*}
     \begin{split}
         G+(u)= H = G^\perp + (h) & =D-G +(\eta) +(h)\\ \Rightarrow  (\eta)+(hu^{-1})&= 2G-D\\
         \Rightarrow (\omega)&=2G-D, 
     \end{split}
 \end{equation*}
where $\omega=hu^{-1} \eta$ is a canonical divisor. This proves the converse direction.\\

Moreover, observe that, in this case, we have
\begin{equation*}
    \begin{split}
       \deg(2G-D) &=\deg (\omega)= 2g-2\\
       \Rightarrow \, \deg G&= \frac{1}{2}(\deg D+2g-2).
    \end{split}
\end{equation*}

\item[(c)] If the code $ C_{\mL}(D, G)$ is iso-dual, then by (b), it follows that $\deg G= \frac{1}{2}(\deg D-2)$, as $g=0$ for a rational function field.

Conversely, if $\deg G= \frac{1}{2}(\deg D-2)$, then $\deg (2G-D)=2g-2=-2<0$, hence we know that $l(2G-D)=0$. Thus, $2G-D$ is a canonical divisor. Then we are done by (a).
\end{itemize}
   
\end{proof}

\begin{rmk}
     The converse direction in Theorem \ref{th:iso-dual} is stronger than \cite[Theorem 3.3]{isodualgag}, but we have restrictions on the degree of the associated divisors. Note that part (c) of Theorem \ref{th:iso-dual} coincides with \cite[Theorem 3.4]{isodualgag}, but our proof is a consequence of parts (a) and (b) of Theorem \ref{th:iso-dual}.
\end{rmk}

As an immediate corollary, we have the following criterion for self-duality of generalized geometric Goppa codes.
\begin{corollary}
   Let $C_{\mL}(D, G)$ be a generalized geometric Goppa code defined over a function field $F$ of genus $g$. Assume that $2g-2< \deg G < \deg D$ and $\deg D > 2g+2$.  Then  $C_{\mL}(D, G)$ is self-dual if and only if there exists a Weil differential $\eta$ such that $(\eta)= 2G-D$ with $v_{P_i}((\eta))=-1$ and $res_{P_i}(\eta)=1$ for $1 \leq i \leq s$.
\end{corollary}

\section{Lifting iso-dual GGG codes on function field extensions}
In this section, we lift iso-dual generalized geometric Goppa codes on function field extensions following \cite{CPQT}.\par
 Let $M/F$ be a finite separable extension of function fields over $\F_q$ of degree $m \geq 2$ with genera $g_M$ and $g_F$ respectively. Let $\{P_1, \ldots, P_s\}$ and $\{Q_1, \ldots, Q_r\}$ are disjoint set of places of $F$. Let $(\beta_1, \ldots, \beta_r) \in \mathbb{Z}^r$ be a non-zero $r$-tuple and consider the divisors $D= P_1 + \cdots P_s$, \, $G= \beta_1Q_1 + \cdots +\beta_r Q_r$ of $F$. Recall that, for each $i=1, \ldots, s$, $\mO_{P_i} \subset F$ denotes the valuation ring of $P_i$ and $F_{P_i}:=\mathcal{O}_{P_i}/P_i$ denotes the associated residue class field. Recall that $\mc{V}:=\prod_{i=1}^s F_{P_i}$ is a $\F_q$-vector space of dimension $\deg D= \sum k_i$, where $k_i$ is the degree of each $P_i$. Recall the generalized geometric Goppa code $C_{\mL}(D,G)$ is the image of the evaluation map $\text{ev}:\mL(G) \rar \mc{V},~ x \mapsto (x(P_i))_{i=1}^s$ defined in \eqref{GGG_1}.

 Suppose that the following condition holds:
    \begin{enumerate}[label=(\alph*)]
        \item $P_i$ splits completely in $M/F$ for $i=1, \ldots,s$

        \item $M/F$ is unramified outside $\textrm{Supp}(G)$

        \item for each $Q \in \textrm{Supp}(G)$, the different exponent $d(S|Q)$ is even for all $S$ lying over $Q$.
    \end{enumerate}

For each $i=1, \ldots, s$, since $P_i$ splits completely, there are exactly $m=[M:F]$ distinct places of $M$ lying over $P_i$. Let $R_{i,1}, \ldots, R_{i,m}$ be all distinct places of $M$ lying over $P_i$. Observe that $e(R_{i,j}|P_i)=1$ for all $j=1, \ldots, m$, as $P_i \notin \textrm{Supp}~(G)$. Let $\mO_{R_{i,j}} \subset M$ denote the valuation ring of $R_{i,j}$ and $M_{R_{i,j}}:=\mathcal{O}_{R_{i,j}}/R_{i,j}$ denote the associated residue class field. By the Fundamental Equality (see \cite[Theorem 3.1.11]{St}), it follows that $M_{R_{i,j}}=F_{P_i}$ for all $i=1, \ldots, s, \, j=1, \ldots, m$. Thus, we also get that $\deg R_{i,j}=\deg P_i$. We now consider the $\F_q$-vector space $$\mc{W}:=\prod_{\substack{1 \leq i \leq s\\ 1 \leq j \leq m }} M_{R_{i,j}}$$ of dimension $m \sum k_i$.
Condition (c) above implies that $\widehat{G}= \Con_{M/F}(G) + \frac{1}{2}{\rm Diff}(M/F)$ is a well-defined divisor of $M$ (see \eqref{eq:different}). Consider the divisor $\t{D}=\Con_{M/F}(D)$ of $M$. Then we define the lift of the generalized geometric Goppa code $C_{\mL}(D,G)$ to be the generalized geometric code $C_{\mL}(\t{D},\widehat{G})$, which is the image of the evaluation map
\begin{equation*}
     \text{ev}:\mL(\widehat{G}) \rar \mc{W},~
     x \mapsto (x(R_{i,j}))_{\substack{1 \leq i \leq s\\ 1 \leq j \leq m }}.
   \end{equation*}

The following theorem shows that, under certain assumptions on the degree of divisors, the lift of an iso-dual generalized geometric Goppa code is also iso-dual, generalizing \cite[Theorem 4.1]{CPQT}.

\begin{theorem}\label{th:lifted}
   Let $M/F$ be a finite separable extension of function fields over $\F_q$ of degree $m \geq 2$ with genera $g_M$ and $g_F$ respectively. Let $\{P_1, \ldots, P_s\}$ and $\{Q_1, \ldots, Q_r\}$ be disjoint set of places of $F$ such that 
    \begin{enumerate}[label=(\alph*)]
        \item $P_i$ splits completely in $M/F$ for $i=1, \ldots,s$

        \item $M/F$ is unramified outside $\textrm{Supp}(G)$

        \item for each $Q \in \textrm{Supp}(G)$, the different exponent $d(S|Q)$ is even for all $S$ lying over $Q$.
    \end{enumerate}
     Let $(\beta_1, \ldots, \beta_r) \in \mathbb{Z}^r$ be a non-zero $r$-tuple and consider the divisors $D= P_1 + \cdots P_s$, \, $G= \beta_1Q_1 + \cdots +\beta_r Q_r$ of $F$.
  Let $$\deg G> g_F-1+\frac{g_M-1}{m}$$ and $$\deg D> \max\{2g_F+2,\frac{2 g_M+2}{m},\deg G-g_F+1+\frac{g_M-1}{m}\}.$$ Furthermore, consider the divisors $\widehat{G}= \Con_{M/F}(G) + \frac{1}{2}{\rm Diff}(M/F)$ and $\t{D}=\Con_{M/F}(D)$ of $M$. If the generalized geometric Goppa code $C_\mL(D,G)$ is iso-dual, then so is $C_\mL(\t{D}, \widehat{G})$.
\end{theorem}

\begin{proof}
By the Hurwitz genus formula \cite[Theorem 3.4.13]{St}, we have that 
\begin{equation}\label{eq:deg_different}
\deg {\rm Diff}(M/F)=(2g_M-2)-m(2g_F-2). 
\end{equation}
Since ${\rm Diff}(M/F)\geq 0$, we have that 
$$ g_M-1\geq m(g_F-1).$$
Hence, by hypothesis,
\begin{equation}\label{eq:basefield}
    \deg D> \deg G + \frac{1}{2m}\deg {\rm Diff}(M/F) \geq \deg G >g_F-1+\frac{g_M-1}{m} \geq 2g_F-2.
\end{equation}
For any divisor $A$ of $F$, we know that $\deg \t{A}= m \deg A$, by \eqref{eq:degree}. Thus, by hypothesis, 
\begin{equation}\label{eq: D lowerbound}
    \deg \t{D}= m\deg D>2g_M+2.
\end{equation} 
Moreover, we have
 \begin{equation}\label{eq:upperbound}
    \begin{split}
        \deg \widehat{G}&=m\deg G+\frac{1}{2}\deg{\rm Diff}(M/F)\\
        &= m\deg G+(g_M-1)-m(g_F-1) \ \ (\textrm{ by \eqref{eq:deg_different}})\\
        &=m(\deg G-g_F+1+\frac{g_M-1}{m})<m \deg D=\deg \t{D}.
    \end{split}
\end{equation}
and 
\begin{equation}\label{eq:lowerbound}
    \begin{split}
        \deg \widehat{G}&=m\deg G+\frac{1}{2}\deg{\rm Diff}(M/F)\\
        &= m\deg G + (g_M-1)-m(g_F-1)\\
        &>m(g_F-1)+(2g_M-2)-m(g_F-1)=2g_M-2.
    \end{split}
\end{equation}
Suppose that the generalized geometric Goppa code $C_\mL(D,G)$ is iso-dual. Since the conditions on degrees of $D$ and $G$ are satisfied by the hypothesis and \eqref{eq:basefield}, by Theorem \ref{th:iso-dual}, there is a a Weil differential $\eta$ of $F$ such that $(\eta)=2G-D$ with $v_P((\eta))=-1$ for all $P \in \textrm{Supp}~(D)$. Let $\t{\eta}= {\rm Cotr}\,(\eta)$ be the cotrace of $\eta$ in $M$. Hence, by \eqref{eq:cotrace}, we have
    \begin{equation*}
    \begin{split}
         (\t{\eta})&={\rm Con}_{M/F} ((\eta))+ {\rm Diff} \, (M/F)\\
         &={\rm Con}_{M/F}(2G-D)+ {\rm Diff} \, (M/F)\\
         &={\rm Con}_{M/F}(2G)-{\rm Con}_{M/F}(D)+ {\rm Diff} \, (M/F)\\
         &=2\widehat{G}-\t{D}.
    \end{split}
        \end{equation*}
We have $\textrm{Supp}(G) \, \cap \textrm{Supp}(D)= \emptyset$, which implies that ${\rm Supp} (2\widehat{G}) \, \cap {\rm Supp}(\t{D})= \emptyset.$ Moreover, $d(R_{i,j}|P_i)=e(R_{i,j}|P_i)-1=0$, follows from $(b)$. Thus, \[v_{R_{i,j}}((\t{\eta}))=-1\]
for all $i=1, \ldots, s, j=1, \ldots, m$. Since the conditions of Theorem \ref{th:iso-dual} on degrees of $\t{D}$ and $\widehat{G}$ are satisfied by \eqref{eq: D lowerbound}, \eqref{eq:upperbound} and \eqref{eq:lowerbound}, choosing $$x=(\operatorname{res}_{R_{1,1}}(\t{\eta}), \ldots, \operatorname{res}_{R_{s,m}}(\t{\eta})) \in \mc{W},$$ by Theorem \ref{th:iso-dual} (a), we have that the generalized geometric Goppa code $C_\mL(\t{D}, \widehat{G})$ is iso-dual.

\smallskip


 \end{proof}
As an immediate corollary, we have the following:
 \begin{corollary}\label{cor:parameters}
     Under the assumptions of Theorem \ref{th:lifted}, the lifted iso-dual code $C_\mL(\t{D}, \widehat{G})$ has parameters $[\t{n}, \t{k}, \t{d}]$, where \[\t{n}=ms, \ \t{k}= \frac{m~\deg D}{2},\ \t{d}\geq ms-\frac{1}{2}(m~\deg D-2g_M+2).\].
 \end{corollary}
 \begin{proof}
     By definition, we see that $\t{n}=ms$. Since $C_\mL(\t{D}, \widehat{G})$ is iso-dual, $$\t{k}=\frac{1}{2}\dim \mc{W}=\frac{m~\deg~D}{2}.$$
     Moreover, by Theorem \ref{th:iso-dual}, we have that $\deg \widehat{G}= \frac{1}{2}(\deg \t{D}+2g_M-2)= \frac{1}{2}(m \deg D+2g_M-2)$. Thus, from \cite[Proposition 1]{GGG}, it follows that  $$\t{d}\geq ms-\deg \widehat{G}=ms-\frac{1}{2}(m \deg D+2g_M-2).$$
 \end{proof}

 \subsection{Elementary abelian p-extensions}
 Let $f$ be a polynomial over $\mathbb{F}_{q^r}$, $r \geq 1$, of degree $m'>0$ such that $gcd(q,m')=1$ and let $ 0 \neq \mu \in \mathbb{F}_{q^r}$. Suppose the polynomial $T^q+ \mu T \in \mathbb{F}_{q^r}[T]$ splits completely into linear factors over $K=\mathbb{F}_{q^r}$. A function field of the form $F=K(x,y)$ where \begin{equation}\label{eq_elem}
     y^q+\mu y=f(x),
 \end{equation}
is called an elementary abelian $p$-extension of $K(x)$. From \cite[Proposition 6.4.1]{St}, we have that the extension $F/K(x)$ is of degree $q$ and genus $g=\frac{(q-1)(m'-1)}{2}$. The only pole $P_{\infty}$ of $x$ in $K(x)$ is totally ramified in $F$ and the extension $F/K(x)$ is unramified outside the place $P_{\infty}$. The different exponent $d(Q_{\infty}|P_{\infty})=(q-1)(m'+1)$ is even, where $Q_{\infty}$ is the only place of $F$ lying over $P_{\infty}$.
Now we construct iso-dual codes on elementary abelian $p$-extensions.
\begin{proposition}
    Let $r \in \mathbb{N}$ and $m'>3$. Consider the elementary abelian $p$-extension $F=\mathbb{F}_{q^r}(x,y)$ as in \eqref{eq_elem}. Let $P_1,\dots, P_s$ be $s$ distinct places of $\mathbb{F}_{q^r}(x)$ different from $P_{\infty}$ and $D=\sum_{i=1}^s P_i$ and $G=\frac{1}{2}(deg~D-2)P_{\infty}$. Assume that $\deg~D>\frac{(q-1)(m'-1)+2}{q}$ and is even. Also, assume that for each $1 \leq i \leq s$, $P_i$ splits completely in $F/K(x)$. Then there exists an iso-dual generalized geometric Goppa code over $F$ with parameters $[\tilde{n},\tilde{k}, \tilde{d}]$ where $$\tilde{n}=qs,~\tilde{k}=\frac{q~\deg D}{2}, ~\tilde{d} \geq qs-\frac{(q~\deg D-(q-1)(m'-1)+2)}{2}.$$
\end{proposition}
\begin{proof}
  By Theorem \ref{th:iso-dual} (c), $C_\mL(D,G)$ is iso-dual. Then it is easy to check that all conditions of Theorem \ref{th:lifted} hold.  
\end{proof}
 \section{Examples}
In this section, we illustrate our results with several examples. Also see \cite[Section 4]{isodualgag} for interesting examples of iso-dual GAG codes.
 
 \begin{example}
     Consider the rational function field $F=\F_2(x)$. Let $P_1$ denote the rational place of $\F_2(x)$ corresponding to the polynomial $x+1$ and $P_2$ denote the place of degree $3$ of $\F_2(x)$ corresponding to the polynomial $x^3+x+1$. Let us denote by $P_\infty$ the only pole of $x$ in $\F_2(x)$. Let $D=P_1+P_2$ and $G=P_\infty$ be divisors of $F$. Then $\deg D=4$ and $\deg G=1=\frac{1}{2}(\deg D-2)$. Hence, by Theorem \ref{th:iso-dual}(c), we see that the code $C_\mL(D,G)$ is iso-dual.

     Now, consider the elementary abelian $2$-extension $M=\F_2(x,y)$ of $\F_2(x)$, where $y^2+ y=x^3+1$. Degree of the finite separable extension $M/F$ is $2$ and genus of $M$ is $1$. Note that $T^2+T$ factors into distinct linear factors over the residue class field of $P_1$. By \cite[Theorem 3.3.7]{St}, these correspond to two places $R_{1,1}, R_{1,2}$ in $M$ lying over $P_1$. For $P_2$, the residue class field is $\F_8=\F_2(\alpha)$, where $\alpha$ satisfies $\alpha^3+\alpha+1=0$. Then the polynomial $T^2+T+ \alpha=0$ has two roots $\alpha^2, \alpha^2+1$ in $\F_8$. Again using \cite[Theorem 3.3.7]{St}, these correspond to two places $R_{2,1}, R_{2,2}$ lying over $P_2$. Hence, we see that $P_1, P_2$ split completely in $M/F$. Also, we know that $M/F$ is unramified outside $\textrm{Supp}(G)$. Let $Q_\infty$ denote the place of $M$ lying over $P_{\infty}$. Note that $\text{Diff}(M/F)=4Q_\infty$. Moreover, observe that $1=\deg G>g_F-1+\frac{g_M-1}{m}=-1$ and $4=\deg D> \max\{2g_F+2,\frac{2 g_M+2}{m},\deg G-g_F+1+\frac{g_M-1}{m}\}=\max\{2, 2, 5/2\}$. Let us consider the divisors $\widetilde{D}=\text{Con}_{M/F}(D)=R_{1,1}+R_{1,2}+R_{2,1}+R_{2,2}$ and $\widehat{G}=\text{Con}_{M/F}(G)+\frac{1}{2}\text{Diff}(M/F)=4Q_\infty$. Since all the conditions of Theorem \ref{th:lifted} are satisfied, we obtain an iso-dual generalized geometric Goppa code $C_\mL(\widetilde{D}, \widehat{G})$ over $M$. By Corollary \ref{cor:parameters}, it has parameters $[4, 4, \t{d}]$ where $\t{d} \geq 0$.
 \end{example}

 \begin{example}
     Let us consider the rational function field $F=\F_3(x)$. Let $P_1$ denote the place of degree $2$ of $\F_3(x)$ corresponding to the polynomial $x^2+x+2$ and $P_2$ denote the place of degree $2$ of $\F_3(x)$ corresponding to the polynomial $x^2+2x+2$. Let us denote by $P_\infty$ the only pole of $x$ in $\F_3(x)$. Let $D=P_1+P_2$ and $G=P_\infty$ be divisors of $F$. Then $\deg D=4$ and $\deg G=1=\frac{1}{2}(\deg D-2)$. Again, by Theorem \ref{th:iso-dual}(c), we see that the code $C_\mL(D,G)$ is iso-dual.

      We consider the elementary abelian $3$-extension $M=\F_3(x,y)$ of $F$, where $y^3- y=x^4+1$. Degree of the finite separable extension $M/F$ is $3$ and genus of $M$ is $3$.  The residue class field of $P_1$ is $\F_9=\F_3(\alpha)$, where $\alpha$ satisfies $\alpha^2+\alpha+2=0$. Note that $T^3-T$ factors into distinct linear factors over the residue class field of $P_1$. By \cite[Theorem 3.3.7]{St}, these correspond to three places $R_{1,1}, R_{1,2}$ and $R_{1,3}$ in $M$ lying over $P_1$. For $P_2$, the residue class field is $\F_9=\F_3(\beta)$, where $\beta$ satisfies $\beta^2+2\beta+2=0$. Then the polynomial $T^3-T$ has three distinct roots in $\F_9$. Again using \cite[Theorem 3.3.7]{St}, these correspond to three places $R_{2,1}, R_{2,2}, R_{2,3}$ lying over $P_2$. Hence, we see that $P_1, P_2$ split completely in $M/F$. Also, we know that $M/F$ is unramified outside $\textrm{Supp}(G)$. Let $Q_\infty$ be the place of $M$ lying over $P_{\infty}$. Note that $\text{Diff}(M/F)=10Q_\infty$.  Again, note that $1=\deg G>g_F-1+\frac{g_M-1}{m}=-1/3$ and $4=\deg D> \max\{2g_F+2,\frac{2 g_M+2}{m},\deg G-g_F+1+\frac{g_M-1}{m}\}=\max\{2, 8/3, 8/3\}$. Following Theorem \ref{th:lifted}, let us consider the divisors $\widetilde{D}=\text{Con}_{M/F}(D)=R_{1,1}+R_{1,2}+R_{1,3}+R_{2,1}+R_{2,2}+R_{2,3}$ and $\widehat{G}=\text{Con}_{M/F}(G)+\frac{1}{2}\text{Diff}(M/F)=8Q_\infty$. Hence, the code $C_\mL(\widetilde{D}, \widehat{G})$ is iso-dual generalized geometric Goppa code over $M$. By Corollary \ref{cor:parameters}, it has parameters $[6, 6, \tilde{d}]$, where $\tilde{d} \geq 2$.
 \end{example}

  \begin{example}
     Consider the rational function field $F=\F_2(x)$. Let $P_1$ denote the rational place of $\F_2(x)$ corresponding to the polynomial $x+1$ and $P_2$ denote the place of degree $3$ of $\F_2(x)$ corresponding to the polynomial $x^3+x+1$. Let us denote by $P_\infty$ the only pole of $x$ in $\F_2(x)$. Let $D=P_1+P_2$ and $G=P_\infty$ be divisors of $F$. Then $\deg D=4$ and $\deg G=1=\frac{1}{2}(\deg D-2)$. Hence, by Theorem \ref{th:iso-dual}(c), we see that the code $C_\mL(D,G)$ is iso-dual.

     Now, consider the elementary abelian $2$-extension $M=\F_2(x,y)$ of $F$, where $y^2+ y=(x^3+1)(x^6+x^2+1)$. Degree of the finite separable extension $M/F$ is $2$ and genus of $M$ is $4$. Note that $T^2+T$ factors into distinct linear factors over the residue class field of $P_1$. By \cite[Theorem 3.3.7]{St}, these correspond to two places $R_{1,1}, R_{1,2}$ in $M$ lying over $P_1$. Similarly for $P_2$, using \cite[Theorem 3.3.7]{St}, these correspond to two places $R_{2,1}, R_{2,2}$ lying over $P_2$. Hence, we see that $P_1, P_2$ split completely in $M/F$. Also, we know that $M/F$ is unramified outside $\textrm{Supp}(G)$. Let $Q_\infty$ be the place of $M$ lying over $P_{\infty}$. Note that $\text{Diff}(M/F)=10 Q_\infty$. Following Theorem \ref{th:lifted}, let us consider the divisors $\widetilde{D}=\text{Con}_{M/F}(D)=R_{1,1}+R_{1,2}+R_{2,1}+R_{2,2}$ and $\widetilde{G}=\text{Con}_{M/F}(G)+\frac{1}{2}\text{Diff}(M/F)=7Q_\infty$. Then $\deg~\widetilde{D}=8$ and $deg~\widetilde{G}=7$. This gives $2g_M+2=10 \not < 8=\deg~\widetilde{D}$. So, the conditions on degree of Theorem \ref{th:lifted} does not hold. However, the code $C_{\mathcal{L}}(\widetilde{D}, \widehat{G})$ has basis $\{(1,1,1,1),(1,1,\alpha,\alpha),(1,1,\alpha^2,\alpha^2),(1,1,\alpha+1,\alpha+1)\}$ over $\mathbb{F}_2$, where $\alpha^3+\alpha+1=0$. It is easy to see that the code $C_{\mathcal{L}}(\widehat{D}, \widetilde{G})$ is self-dual.
 \end{example}

 \section{Self-dual Generalized Algebraic-Geometric codes}
 A linear code $C$ over $\mathbb{F}_q$ is called self-dual if $C=C^{\perp}$. In this section, we study self-duality of generalized AG codes.\par
 Recall the inner product \eqref{eq_inn} on $\mathcal{V}$. Under this inner product, we have $$C_{\mathcal{L}}(G,D)^{\perp}=C_{\Omega}(D,G).$$ In \cite{DM}, the authors generalized this inner product for the $\mathbb{F}_q$-vector space $\mathcal{C}=\prod_{i=1}^s C_i$ as follows:\par
  For each $i=1,\dots,s$, fix an almost self-dual $\mathbb{F}_q$-basis $B_i=\{u_1^{(i)},u_2^{(i)},\dots, u_{k_i}^{(i)}\}$ of $F_{P_i}$ and let $t_j^{(i)}:=Tr_{F_{P_i}/\mathbb{F}_q}((u_j^{(i)})^2)$. Then the generalized duality is defined as: for any $a=(a_{1,1},\dots,a_{1,n_1},\dots, a_{s,1}, \dots, a_{s,k_s})$, $b=(b_{1,1},\dots,b_{1,n_1},$ $\dots, b_{s,1}, \dots, b_{s,k_s}) \in \mathcal{C}$, we define
  $$a \star b=\sum_{i=1}^s \sum_{j=1}^{k_i}t_j^{(i)} a_{i,j}b_{i,j}.$$
   We now recall the construction of GAG codes through differentials \cite{DM}. With the same notations as section $2.2$, we have $D=P_1+\dots+P_s$. Consider the map
   $$\gamma: \Omega(G-D) \xrightarrow{rs} \prod_{i=1}^s \mathcal{O}_{P_i}/P_i \xrightarrow{\prod_{i=1}^s \pi_i} \prod_{i=1}^s C_i,~\omega \mapsto (res_{P_i}(\omega))_{i=1}^s \mapsto (\pi_i(res_{P_i}(\omega)))_{i=1}^s.$$
   Then $C_{\Omega}(P_1,\dots,P_s,G,C_1,\dots, C_s):=\gamma(\Omega(G-D))$ defines an $\mathbb{F}_q$-linear code. This code has the property that with respect to the inner product $\star$ on $\mathcal{C}$ \cite[Theorem 9]{DM} $$C_{\mathcal{L}}(P_1,\dots,P_s,G,C_1,\dots, C_s)^{\perp}=C_{\Omega}(P_1,\dots,P_s,G,C_1,\dots, C_s).$$
Also, $$\dim~C_{\mathcal{L}}(P_1,\dots,P_s,G,C_1,\dots, C_s)+ \dim~C_{\Omega}(P_1,\dots,P_s,G,C_1,\dots, C_s)=\dim~\mathcal{C}=\sum_{i=1}^s k_i.$$
If $C_{\mathcal{L}}(P_1,\dots,P_s,G,C_1,\dots, C_s)$ is self-dual with respect to $\star$, then we call it $\star$-self-dual.\\~\\
 With respect to the fixed basis $B_i=\{u_1^{(i)},u_2^{(i)},\dots, u_{k_i}^{(i)}\}$ of $F_{P_i}$, $i=1,\dots,s$. Any element $\alpha \in F_{P_i}$ can be written as
 $$\alpha=\sum_{j=1}^{k_i} Tr_{F_{P_i}/\mathbb{F}_q} \Bigg( \alpha \frac{u_j^{(i)}}{t_j^{(i)}}\Bigg) u_j^{(i)}.$$
 Thus, for $x \in \mathcal{L}(G)$ and $\omega \in \Omega(G-D)$ we have 
$$x(P_i)=x_1^{(i)}u_1^{(i)}+x_2^{(i)}u_2^{(i)}+\dots+x_{k_i}^{(i)}u_{k_i}^{(i)},$$
 $$res_{P_i}(\omega)=\omega_1^{(i)}u_1^{(i)}+\omega_2^{(i)}u_2^{(i)}+\dots+\omega_{k_i}^{(i)}u_{k_i}^{(i)},$$
 where
 $$x_j^{(i)}=Tr_{F_{P_i}/\mathbb{F}_q} \Bigg(x(P_i)\frac{u_j^{(i)}}{t_j^{(i)}} \Bigg) \text{ and }\omega_j^{(i)}=Tr_{F_{P_i}/\mathbb{F}_q} \Bigg(res_{P_i}(\omega) \frac{u_j^{(i)}}{t_j^{(i)}} \Bigg).$$

 So, \eqref{eq_inn} gives
  \begin{equation} \label{eq_GAG}
      \langle ev(x),rs(\omega) \rangle:=\sum_{i=1}^s \sum_{j=1}^{k_i} t_j^{(i)} x_j^{(i)} \omega_j^{(i)}=0.\\
  \end{equation}
  Let $M_i$ be the generator matrix of $C_i$ in standard form, that is, $M_i=(I_{k_i}~\overline{M}_i)$. Further, let $c_j^{(i)}$ denote the $j$th row of $M_i$. For each $1 \leq i \leq s$ and $1 \leq j \leq k_i$, let $\pi_i(u_j^{(i)})=c_j^{(i)}$. Let $\bar{c}_j^{(i)}$ be the vector consisting of the last $n_i-k_i$ components of $c_j^{(i)}$. Then codewords corresponding to $x$ and $\omega$ are
   $$a=(x_1^{(1)},\dots, x_{k_1}^{(1)},\sum_{j=1}^{k_1} x_j^{(1)} \bar{c}_j^{(1)},\dots,x_1^{(s)},\dots, x_{k_s}^{(s)},\sum_{j=1}^{k_s} x_j^{(s)} \bar{c}_j^{(s)})$$
   and
   $$b=(\omega_1^{(1)},\dots, \omega_{k_1}^{(1)},\sum_{j=1}^{k_1} \omega_j^{(1)} \bar{c}_j^{(1)},\dots,\omega_1^{(s)},\dots, \omega_{k_s}^{(s)}, \sum_{j=1}^{k_s} \omega_j^{(s)} \bar{c}_j^{(s)}).$$

In the following theorem, we prove that under certain conditions self-duality of generalized geometric Goppa code $C_{\mathcal{L}}(D,G)$ implies $\star$-self-duality of generalized AG code  $C_{\mathcal{L}}(P_1,\dots,P_s,G,\hspace{2cm}\\ C_1,\dots, C_s)$ and vice-versa.\\
 \begin{theorem}
     Assume $\deg~G < \sum_{i=1}^s k_i$ and $\sum_{i=1}^s k_i$ is even. Then $C_{\mathcal{L}}(D,G)$ is self-dual if and only if $C_{\mathcal{L}}(P_1,\dots,P_s,G,C_1,\dots, C_s)$ is $\star$-self-dual.
 \end{theorem}
 \begin{proof}
   It follows from $deg~G < \sum_{i=1}^s k_i=deg~D$ and $\pi$ being an isomorphism that 
   $$\dim C_{\mathcal{L}}(D,G)=k=l(G)=\dim C_{\mathcal{L}}(P_1,\dots,P_s,G,C_1,\dots, C_s).$$
   To prove the theorem let us first suppose that $C_{\mathcal{L}}(D,G)$ is self-dual, then $k=\frac{\sum_{i=1}^s k_i}{2}$. Also,
   \begin{align*}
   \dim C_{\Omega}(P_1,\dots,P_s,G,C_1,\dots, C_s)
    &=\dim C_{\mathcal{L}}(P_1,\dots,P_s,G,C_1,\dots, C_s)^{\perp}\\ &=\sum_{i=1}^sk_i-\dim C_{\mathcal{L}}(P_1,\dots,P_s,G,C_1,\dots, C_s)\\
    &=\sum_{i=1}^s k_i -k=2k-k=k\\
    &=\dim C_{\mathcal{L}}(P_1,\dots,P_s,G,C_1,\dots, C_s).\\
   \end{align*}
   Now it is enough to show that $C_{\mathcal{L}}(P_1,\dots,P_s,G,C_1,\dots, C_s)\subseteq C_{\Omega}(P_1,\dots,P_s,G,C_1,\dots, C_s)$. Let $x \in \mathcal{L}(G)$ and $\omega \in \Omega(G-D)$. 
 Then it follows from \eqref{eq_GAG} that 
 $$a \star b=\sum_{i=1}^s \sum_{j=1}^{k_i}t_j^{(i)} x_{j}^{(i)} \omega_{j}^{(i)}=\langle ev(x),rs(\omega) \rangle=0.$$  This shows that $$C_{\mathcal{L}}(P_1,\dots,P_s,G,C_1,\dots, C_s)\subseteq C_{\Omega}(P_1,\dots,P_s,G,C_1,\dots, C_s)=C_{\mathcal{L}}(P_1,\dots,P_s,G,C_1,\dots, C_s)^{\perp}.$$
 From above, the dimension of the two codes are equal. Hence, the equality follows. \par 
 Converse part follows by the same arguments.
 \end{proof}

\bibliographystyle{alpha}
\bibliography{GAG}
\end{document}